\documentclass[12pt]{amsart}
\usepackage{amssymb}
\usepackage{amsmath}
\usepackage{bbm}
\usepackage[dvips]{graphicx}
\usepackage{hyperref}
\usepackage[dvipsnames]{xcolor}
\usepackage{todonotes}

\usepackage{tikz}
\usetikzlibrary{arrows}
\usetikzlibrary{decorations.markings}
\tikzset{invmidarrow/.style={postaction=decorate,decoration={markings,mark={between positions 0 and 1 step #1
with {\arrow{latex reversed}}}}}}
\usetikzlibrary{shapes.geometric}

\numberwithin{equation}{section}

\newcommand{\be}{\begin{equation}}
\newcommand{\ee}{\end{equation}}
\newcommand{\id}[1]{\mathbbm{1}_{#1}}
\newcommand{\iden}{\mathbbm{1}}
\newcommand{\imag}{\iota}
\newcommand{\md}{monopole-dimer}
\newcommand{\mdm}{monopole-dimer model}

\DeclareMathOperator{\sgn}{sgn}
\newcommand{\cO}{\mathcal O}
\newcommand{\dicot}[1]{\mathcal{#1}}
\newcommand{\dicotpf}[1]{\mathcal{#1}}
\newcommand{\dicotmat}[1]{\mathcal{#1}}
\newcommand{\dicotwt}[1]{\text{wt}(#1)}

\makeatletter
\def\Ddots{\mathinner{\mkern1mu\raise\p@
\vbox{\kern7\p@\hbox{.}}\mkern2mu
\raise4\p@\hbox{.}\mkern2mu\raise7\p@\hbox{.}\mkern1mu}}
\makeatother

\newtheorem{thm}{Theorem}[section]
\newtheorem{prop}[thm]{Proposition}
\newtheorem{cor}[thm]{Corollary}
\newtheorem{lem}[thm]{Lemma}
\newtheorem{defn}[thm]{Definition}
\newtheorem{rem}[thm]{Remark}

\newtheorem{eg}[thm]{Example}
\newtheorem{ques}[thm]{Question}

\title[Squareness for the Monopole-Dimer model]
{Squareness for the Monopole-Dimer model}
\author{Arvind Ayyer}
\address{Arvind Ayyer, Department of Mathematics, 
Indian Institute of Science, Bangalore  560012, India.}
\email{arvind@iisc.ac.in}


\date{\today}

\begin{document}

\begin{abstract}
The monopole-dimer model introduced recently 
is an exactly-solvable signed generalisation of the dimer model.
We show that the partition function of the monopole-dimer model on a graph invariant under a fixed-point free involution is a perfect square. We give a combinatorial interpretation of the square-root of the partition function for such graphs in terms of a monopole-dimer model on a new kind of graph with two types of edges which we call a dicot. The partition function of the latter can be written as a determinant, this time of a complex adjacency matrix. This formulation generalises T. T. Wu's assignment of imaginary orientation for the grid graph to planar dicots. As an application, we compute the partition function for a family of non-planar dicots with positive weights.
\end{abstract}

\maketitle

\section{Introduction}
The dimer model (also known as perfect matchings or $1$-factors) for planar graphs is a well-studied exactly solvable statistical-mechanical model whose partition function can be evaluated as a Pfaffian \cite{kasteleyn1961}. A generalisation to the monomer-dimer model with boundary mono\-mers has been recently obtained by Giuliani-Jauslin-Lieb \cite{giuliani-jauslin-lieb-2016}. In another direction, the monopole-dimer model was introduced recently \cite{ayyer2015} as a generalisation of the double-dimer model. The partition function of the \mdm{} was shown to be given by a determinant for planar graphs. Further, it was found that the partition function of the model on some planar graphs such as the cycle graph $C_{4n}$ and the $2m \times 2n$ grid graph could be expressed as a perfect square. 

In this work, we give a combinatorial explanation for the squareness phenomenon for the monopole-dimer model in a general setting using a symmetry argument in the spirit of Jockusch \cite{jockusch-1994} and Ciucu \cite{ciucu-1997}  for the dimer model.
To that end, we will define a generalisation of graphs with two kinds of edges (solid and dashed), which we will call dicots\footnote{short for {\em dicotyledons}, a plant whose embryo contains two leaves.} in Section~\ref{sec:md-dicot}.  We will define the monopole-dimer model for dicots and show that the partition function of the model can be written as a determinant in Section~\ref{sec:md-dicot}. When the dicot is planar, we will show that there exists a generalisation of the Kasteleyn orientation that makes it possible to assign a natural combinatorial weight to the model.

We note that this formulation also generalises Wu's choice of  orientation of 
$\imag =\sqrt{-1}$ for vertical edges \cite{wu1962} for two-dimensional grid graphs. This trick has proved useful in analysing the dimer model on subgraphs of $\mathbb{Z}^2$; see for example \cite{kenyon1997}.

We will consider graphs with a special automorphism in Section~\ref{sec:two-md-models} and prove that the \md{} partition function is a square.

In Section~\ref{sec:families}, we will provide exact results for the partition functions for certain families of dicots and use those results to calculate the free energy in each case. In Sections~\ref{sec:cycles} and ~\ref{sec:rect-grids}, we will explain the squareness for cycles and rectangular grids observed in \cite{ayyer2015}.
We will then consider dicots on the rectangular grid with additional vertical dashed edges in Section~\ref{sec:rect-grids-add-vert}. 
In Section~\ref{sec:wheel}, we will consider a family of non-planar dicots for which the monopole-dimer model is manifestly positive.

\section{Monopole-Dimer Model on Dicots}
\label{sec:md-dicot}
In this section, we will define dicots, the monopole-dimer model on dicots and prove the result about the partition function of the latter. We then consider the special case of planar dicots, define the Kasteleyn orientation for such dicots and prove the analogous result for the monopole-dimer model there.

\noindent
{\bf Remark on notation:} We will always use unaccented symbols for objects associated to graphs and calligraphic symbols for those associated to dicots.

Dicots are generalisations of graphs with two sets of edges.
We will denote the two sets of edges on $V$ by $A$ and $B$. Edges belonging to $A$ will be denoted by solid lines, and those belonging to $B$, by dashed lines. We will also interchangeably use the term solid (resp. dashed) edges to mean $A$-type (resp. $B$-type) edges. 

\begin{defn} \label{def:dicot}
We say that $\dicot{G} = [V,A,B]$ is a {\bf dicot}  if $\dicot{G}$ satisfies the following conditions.
\begin{enumerate}
\item $[V,A \cup B]$ is a bipartite graph.
\item The subgraphs $[V,A]$ and $[V,B]$ are simple, i.e. have no loops or multiple edges.
\end{enumerate}
\end{defn}

In other words, there can be at most two edges between any two vertices in $\dicot{G}$, and if there are two edges, then both must be of different type. 
An oriented dicot is one where each solid edge is assigned a direction.
We will denote vertex weights by $x(v)$ for $v \in V$ and edge weights as $a(v,v') \equiv a(v',v)$ whenever $(v,v') \in A$ and $b(v,v') \equiv b(v',v)$ whenever $(v,v') \in B$.
Dicots with no $B$-type edges are thus bipartite graphs.

Throughout the paper, graphs and dicots will be undirected, both vertex- and edge- weighted.  For the purpose of enumeration, we will prescribe an orientation  just as for the original dimer model. Graphs will be simple (i.e. without loops or multiple edges).
The weights are positive real numbers. 
Unweighted versions are taken care of by setting the weights to be $1$.

\begin{rem}
\label{rem:label}
Throughout the article, unless stated explicitly, graphs and dicots with $n$ vertices will have vertex set $V =\{1,\dots,n\}$ and the orientation for graphs (resp. dicots) associated to edges (resp. solid edges) will point from smaller to larger label. 
\end{rem}

\begin{defn} 
\label{def:completedicot}
The {\bf complete dicot}, denoted $\dicot{D}_{a,b}$, is the dicot on $a+b$ vertices where the separate underlying graphs $[V,A]$ and $[V,B]$ with solid  and dashed edges respectively forms the complete bipartite graph $K_{a,b}$.
\end{defn}

We now generalise the monopole-dimer model for graphs \cite{ayyer2015} to dicots. Recall that monopole-dimer configurations for graphs consist of
directed loops of even length including doubled edges  with the property that every vertex belongs to either zero or two edges.

\begin{defn} \label{def:loopvertconf}
A {\bf monopole-dimer configuration} of $\dicot{G}=[V,A,B]$ with orientation $\cO$ is a monopole-dimer configuration of the graph $[V, A \cup B]$ with the additional proviso that every loop must contain an even number of dashed edges. Let $\mathcal{L}(\dicot{G})$ be the set of monopole-dimer configurations of  $\dicot{G}$.
\end{defn}

The weight of a monopole-dimer configuration $C$ is given by a formula similar to that of the usual monopole-dimer configuration. 
Let $C \in 
\mathcal{L}(\dicot{G})$ be a monopole-dimer configuration on  $\dicot{G}$. Decompose $C$ into loops $(\ell_1,\dots,\ell_q)$ and isolated vertices $(w_1,\dots,w_r)$.

We first define the weight of a loop $\ell$. Let $\ell = (v_1,\dots,v_{2m})$ with edges $e_1,\dots,e_{2m}$ such that $e_j = (v_j,v_{j+1})$ with $v_{2m+1} \equiv v_1$. Then the weight of the loop is given by
\be \label{wtloop}
w(\ell) = - \; \prod_{j=1}^{2 m}  
\begin{cases}
+a(v_j,v_{j+1}) & \text{if $e_j$ is solid and $v_{j} \to v_{j+1}$ in $\cO$,} \\
-a(v_j,v_{j+1}) & \text{if $e_j$ is solid and $v_{j+1} \to v_{j}$ in $\cO$,} \\
\imag \, b(v_j,v_{j+1}) & \text{if $e_j$ is a dashed edge,}
\end{cases}
\ee
where we recall that $\imag = \sqrt{-1}$. 
Notice that the sign of the loop depends on the relative orientations of the solid edges and the number of dashed edges.
The weight of the monopole-dimer configuration $C$ is then given by
\be \label{wtconf}
\dicotwt C = \prod_{j=1}^q w(\ell_j) \prod_{j=1}^r x(w_j).
\ee

\begin{defn}
\label{def:mondim-dicot}
The {\bf monopole-dimer model} on a dicot $\dicot{G}$ with orientation $\cO$ is the collection 
$\mathcal{L}(\dicot{G})$ of dicot monopole-dimer configurations on $\dicot{G}$ with the weight of each configuration given by \eqref{wtconf}.
\end{defn}

The partition function of the monopole-dimer model on  $\dicot{G}$ is
\be \label{dicotpf}
\dicotpf Z(\dicot{G}) = \sum_{C \in \mathcal{L}(\dicot{G})} \dicotwt C.
\ee

We now define an adjacency matrix for a dicot $\dicot{G}$.
\begin{defn} \label{def:matrix}
The {\bf complex adjacency matrix} $\dicotmat K \equiv \dicotmat{K}(\dicot{G})$ associated to  $\dicot{G}$ is the matrix indexed by the vertices of $\dicot{G}$ whose entries are given by
\be 
\dicotmat{K}_{v,v'} = \begin{cases}
x(v) & v'=v \\
a(v,v') + \imag \, b(v,v') & \text{if $v \to v'$ in $\cO$,} \\
-a(v,v') + \imag \, b(v,v') & \text{if $v' \to v$ in $\cO$.}
\end{cases}
\ee
We will treat $a(v,v')$ (resp. $b(v,v')$) as 0 if $v$ and $v'$ are not connected by a
solid (resp. dashed) edge.
\end{defn}

\begin{thm}
\label{thm:dicot-partn-fn}
The partition function of the monopole-dimer model on  $\dicot{G}$ is given by
\be \label{partfn-general}
\dicotpf{Z}(\dicot{G}) = \det \dicotmat{K}(\dicot{G}).
\ee
\end{thm}
See \cite{ayyer2013} for a similar result about the determinant of such matrices.

\begin{proof}
Note that if we do not have any dashed edges, $\dicot{G}$ reduces to a bipartite graph and the proof is a special case of the partition function of the monopole-dimer model on graphs \cite[Theorem 2.5]{ayyer2015}. Our proof here follows the same strategy. 

It should be clear from Definition~\ref{def:dicot}(1) and the complex adjacency matrix in Definition~\ref{def:matrix} that terms in the Leibniz expansion of the determinant of $\dicotmat{K}(\dicot{G})$ that are non-zero correspond to permutations with singletons and even cycles. Since off-diagonal non-zero entries in $\dicotmat{K}(\dicot{G})$ consist of two terms, we also 
expand the determinant in terms of monomials in these variables.
We will first show that these terms are equinumerous with configurations in $\mathcal{L}(\dicot{G})$. We need to show that the signs and weights of these configurations are counted by the determinant.

We will work in the most general setting of the complete dicot $\dicot{D}_{a,b}$ (see Definition~\ref{def:completedicot}) with generic vertex- and edge-weights. 
The result for a generic dicot will follow by setting some of these weights to zero or one.
Let $C \in \mathcal{L}(\dicot{D}_{a,b})$ be given by $C = (\ell_1,\dots,\ell_m;\, w_1,\dots,w_{2p})$, where every vertex $v \in V$ is either one of the isolated vertices $w_j$ or belongs to exactly one loop $\ell_j$. There will be many configurations with the same isolated vertices and the same loop structure which will also contribute because of various choices of solid and dashed edges, as well as the direction of the loops. Since each of the loops $\ell_j$ can be summed separately, it suffices to look at a single loop $\ell = (v_1,\dots,v_{2m})$ with edge $e_j = (v_j,v_{j+1})$.
The entry  corresponding to $e_j$ in $\dicotmat{K} \equiv \dicotmat{K}(\dicot{D}_{a,b})$ is then $\dicotmat{K}_{v_j,v_{j+1}} = z_j$, a complex number. Then the entry corresponding
to the reversed edge $(v_{j+1},v_j)$, is $\dicotmat{K}_{v_{j+1},v_j} = -\bar{z}_j$, it's negative complex conjugate.
If $m=1$, the contribution of $C$ is $z_1 (-\bar{z}_1) = - |z_1|^2$, which is real. 
If $m \geq 2$, we add the contribution of $C$ to its reverse $\widehat{C}$ in $\det \dicotmat{K}$, we get
\[
z_1 \cdots z_{2m} + (-\bar{z}_1) \cdots (-\bar{z}_{2m}) 
= 2 \;\text{Re}(z_1 \cdots z_{2m}).
\]
In both cases, we get a real number, which means that the only loops which contribute are those with an even number of dashed edges, and they contribute with factor two precisely when the loop is nontrivial ($m \geq 2$). These are precisely the allowed terms in the monopole-dimer configuration; see Definition~\ref{def:loopvertconf}.

We will now show that the signs match. For the loop $\ell$ considered above, let $S$ be the set of solid edges with $\# S = 2p$, and among these, $S_-$ be those with opposite sign and $\# S_- = s$. Then the sign of the loop is given, according to \eqref{wtloop}, $\sgn(\ell) =  (-1)^{m-p+1+s}$. 
The corresponding term in the determinant is the part 
in the expansion of $\text{Re}(
\dicotmat{K}_{v_1,v_2} \dots \dicotmat{K}_{v_{2m-1},v_{2m}} \dicotmat{K}_{v_{2m},v_1} )$ which is proportional to $a(e)$ for $e \in S$ and $b(e)$ for $e \notin S$ up to sign. The sign is then given by $i^{2m-2p} (-1)^s = (-1)^{m-p+s}$ times the sign of the corresponding permutation. Finally, it is well-known that the sign of a permutation can be obtained by assigning a negative sign to each even cycle.
Thus, the sign from the determinant matches that from \eqref{wtloop}.
\end{proof}

We illustrate Theorem~\ref{thm:dicot-partn-fn} in the following example.

\begin{eg} \label{eg:complete-dicot-2}
Consider the complete dicot $\dicot{D}_{2,2}$. The vertices are labeled $(1,2,3,4)$, where $\{1,3\}$ and $\{2,4\}$ form the partitions, and the orientation is the standard one (see Remark~\ref{rem:label}).
Let the weight on vertex $j$ be $x_j$ and the edge weights be as shown in Figure~\ref{fig:exampled4}.

\begin{figure}[h]
\begin{center}
\begin{tikzpicture} [>=triangle 45]
\draw (0.0,0.0) node[circle,inner sep=2pt,draw] {1};
\draw (-0.5, -0.5) node [color = red] {$x_1$};
\draw (0.0,3.0) node[circle,inner sep=2pt,draw] {2};
\draw (-0.5, 3.5) node [color = red] {$x_2$};
\draw (3.0,3.0) node[circle,inner sep=2pt,draw] {3};
\draw (3.5, 3.5) node [color = red] {$x_3$};
\draw (3.0,0.0) node[circle,inner sep=2pt,draw] {4};
\draw (3.5, -0.5) node [color = red] {$x_4$};

\draw [-] (0,0.3) -- node [right]  {$a_{12}$} (0,2.7) ;
\draw [-,out=135,in=-135,style=dashed] (-0.3,0) to node [left]  {$b_{12}$} (-0.3,3) ;

\draw [-] (0.3,3) -- node [below]  {$a_{23}$} (2.7,3) ;
\draw [-,out=45,in=135,style=dashed] (0,3.3) to node [above]  {$b_{23}$} (3,3.3) ;

\draw [-] (3,2.7) -- node [left]  {$a_{34}$} (3,0.3) ;
\draw [-,out=-45,in=45,style=dashed] (3.3,3) to node [right]  {$b_{34}$} (3.3,0) ;

\draw [-] (0.3,0) -- node [above]  {$a_{14}$} (2.7,0) ;
\draw [-,out=-45,in=-135,style=dashed] (0,-0.3) to node [below]  {$b_{14}$} (3,-0.3) ;
\end{tikzpicture}
\caption{The complete dicot $\dicot{D}_{2,2}$ with generic vertex- and edge-weights.}
\label{fig:exampled4}
\end{center}
\end{figure}

\noindent
The complex adjacency matrix in the naturally ordered vertex basis is given by
\[
\dicotmat{K}(\dicot{D}_{2,2}) = 
\left(
\begin{matrix}
x_1 & a_{12} + \imag \, b_{12} & 0 & a_{14}+ \imag \, b_{14} \\
-a_{12} + \imag \, b_{12} & x_2 & a_{23} + \imag \, b_{23} & 0 \\
0 & -a_{23} + \imag \, b_{23} & x_3 & a_{34} + \imag \, b_{34} \\
-a_{14} + \imag \, b_{14} & 0 & -a_{34} + \imag \, b_{34} & x_4
\end{matrix}
\right),
\]
and
\begin{align*}
\dicotpf{Z} & (\dicot{D}_{2,2}) = \det \dicotmat{K}(\dicot{D}_{2,2}) \\
=& a_{12}^2 a_{34}^2+a_{12}^2 b_{34}^2+a_{34}^2 b_{12}^2+a_{23}^2 a_{14}^2+a_{23}^2 b_{14}^2 +b_{23}^2 b_{14}^2 +a_{14}^2 b_{23}^2+b_{12}^2 b_{34}^2\\
& +2 a_{12} a_{23} a_{34} a_{14} + 2 a_{12} a_{23} b_{34} b_{14} + 2 a_{12} a_{34} b_{23} b_{14}+2 a_{23} a_{34} b_{12} b_{14} \\
& -2 a_{12} a_{14} b_{23} b_{34}-2 a_{23} a_{14} b_{12} b_{34}-2 a_{34} a_{14} b_{12} b_{23}-2 b_{12} b_{23} b_{34} b_{14}\\
& +a_{12}^2 x_3 x_4+a_{23}^2 x_1 x_4+a_{34}^2 x_1 x_2+a_{14}^2 x_2 x_3+b_{12}^2 x_3 x_4+b_{23}^2 x_1 x_4 \\
& +b_{34}^2 x_1 x_2+b_{14}^2 x_2 x_3+x_1 x_2 x_3 x_4.
\end{align*}
\end{eg}

For the remainder of this section, we will consider the special case of  dicots without multiple edges which can be embedded in a plane.

\begin{defn}
\label{def:pure-dicot}
A dicot $\dicot G = [V,A,B]$ is said to be {\bf pure} if there is at most one edge (either solid or dashed) between any pair of distinct vertices.
\end{defn}

Note that complete dicots $\dicot D_{a,b}$ are never pure.

\begin{defn}
\label{def:planar-dicot}
A dicot $\dicot G = [V,A,B]$ is said to be {\bf planar} if $\dicot G$ is pure, $[V, A \cup B]$ is a planar graph
and every face has an even number of dashed edges.
\end{defn}

In the interest of brevity, we will identify a dicot with its planar embedding. 
We first define the generalisation of the Kasteleyn orientation of a graph for planar dicots. We say that a face in a planar dicot is {\em simple} if it cannot be decomposed into a union of smaller faces. Recall that in a dicot, orientations are only assigned to solid edges.
Note that the number of dashed edges in a planar dicot is always even by Definition~\ref{def:planar-dicot}.

\begin{defn}
\label{def:kast-orient}
Let $\dicot G$ be a planar dicot. Then a {\bf Kasteleyn orientation $\cO$ for $\dicot G$} is one for which the number of clockwise-oriented solid edges plus half the number of dashed edges is odd for every simple face.
\end{defn}

By considering a spanning tree of the dual graph of planar dicots exactly as for planar graphs, it is easy to see that the following holds.

\begin{prop}
\label{prop:kast-orient-exist}
Let $\dicot G$ be a planar dicot. Then there exists a Kasteleyn orientation.
\end{prop}

\begin{proof}
The proof proceeds similar to that for planar graphs. Consider the dual graph $\hat{G}$ of $\dicot G$, where we include the outer (infinite) face and which can have multiple edges. We consider every edge of $G$ crossing a solid (resp. dashed) edge of $\dicot G$ to be solid (resp. dashed). Let $T$ be a spanning forest of $\hat{G}$ consisting of purely solid edges. This is always possible because of the assumptions on $\dicot G$. 

We first orient the solid edges of $\dicot G$ not intersecting with $T$ in any way that we like. We now claim that we can orient the solid edges of $\dicot G$ intersecting with $T$ in a way that we obtain a Kasteleyn orientation according to Definition~\ref{def:kast-orient}. This is always possible to do inductively starting with the leaves of the trees in $T$. In particular, if a simple face is bounded purely by dashed edges, the spanning forest contains a tree which is a  singleton vertex and nothing needs to be done.
\end{proof}

For planar dicots with a Kasteleyn orientation, we give an alternate combinatorial definition of the weight of a loop in the monopole-dimer model.

\begin{cor}
\label{cor:planar dicot}
Let $\dicot G$ be a planar dicot with a Kasteleyn orientation. 
Then the weight of a loop 
$\ell = (v_1,\dots,v_{2m})$ in the monopole-dimer model is given by
\[
w(\ell) = (-1)^{\substack{\text{number of vertices in $V$} \\ \text{enclosed by $\ell$}}} \times \prod_{j=1}^{2m} 
\begin{cases}
a(v_j,v_{j+1}) & \text{if $e_j$ is a solid edge}, \\
b(v_j,v_{j+1}) & \text{if $e_j$ is a dashed edge}.
\end{cases}
\]
\end{cor}

\begin{proof}
We need to show that this definition of the weight is equivalent to that in \ref{wtloop}. The proof follows {\em mutatis mutandis} from \cite[Theorem 3.3]{ayyer2015} by setting $o_f$ to be the number of clockwise-oriented solid edges plus half the number of dashed edges for a face $f$.
\end{proof}

Such planar dicots naturally arise when considering subsets of the two-dimensional grid graphs. 

\begin{rem}
\label{rem:wu-orientation}
Consider the $m \times n$ grid graph and convert it into a dicot by setting all vertical edges to be dashed and orienting the solid edges as those in Figure~\ref{fig:eg-grid-dicot}. This is then a Kasteleyn orientation for a planar dicot. Corollary~\ref{cor:planar dicot} extends the observation of Wu \cite{wu1962} about assigning sign $\imag = \sqrt{-1}$ to vertical edges in grid graphs to obtain a Kasteleyn orientation. 

\begin{figure}[htbp!]
\begin{tikzpicture} [>=triangle 45, scale=1.5]

\foreach \x in {0,1,...,3} {
       \foreach \y in {0,1,...,2} {
            \fill[color=black] (\x,\y) circle (0.05);
       }
}
\draw (3,0) -- (2,0) -- (1,0) -- (0,0);
\draw (0,1) to (1,1) to (2,1) to (3,1);
\draw (3,2) -- (2,2) -- (1,2) -- (0,2);

\draw [-,dashed] (0,0) -- (0,2);
\draw [-,dashed] (1,0) -- (1,2);
\draw [-,dashed] (2,0) -- (2,2);
\draw [-,dashed] (3,0) -- (3,2);

\node [below left] at (0,0) {$1$};
\node [below left] at (1,0) {$2$};
\node [below left] at (2,0) {$3$};
\node [below left] at (3,0) {$4$};

\node [below left] at (0,1) {$8$};
\node [below left] at (1,1) {$7$};
\node [below left] at (2,1) {$6$};
\node [below left] at (3,1) {$5$};

\node [below left] at (0,2) {$9$};
\node [below left] at (1,2) {$10$};
\node [below left] at (2,2) {$11$};
\node [below left] at (3,2) {$12$};

\end{tikzpicture}
\caption{A $4 \times 3$ ``grid dicot'' with the vertices  labelled so that the canonical orientation (see Remark~\ref{rem:label}) is Kasteleyn.}
\label{fig:eg-grid-dicot}
\end{figure}
\end{rem}

We will generalise the planar dicot model in Remark~\ref{rem:wu-orientation} to grids with both solid and dashed vertical edges in Section~\ref{sec:rect-grids-add-vert}.

A natural question is the following.

\begin{ques}
\label{que:positive-pf}
Classify all dicots $\dicot{G}$ with arbitrary weights such that $\dicotpf{Z}(\dicot{G})$ is a positive polynomial,i.e. all monopole-dimer configurations contribute with a positive sign.
\end{ques}

We note that complete dicots do not satisfy this condition.
For example, no choice of orientation of the solid edges in Example~\ref{eg:complete-dicot-2} will make the partition function $\dicotpf{Z} (\dicot{D}_{2,2})$ a positive polynomial.
We will give a nontrivial example of a family of such graphs which are nonplanar in Section~\ref{sec:wheel}.

\section{Squareness for the monopole-dimer model}
\label{sec:two-md-models}
In this section, we will prove that the \md{} model on graphs with a fixed-point-free involution is a perfect square and give a combinatorial interpretation of the square root in terms of a \md{} model on a related dicot.

The monopole-dimer model for graphs $G$ has been defined in \cite{ayyer2015}. 
On bipartite graphs, the \md{} model can be read from Definition~\ref{def:mondim-dicot} on a dicot with no dashed, i.e. $B$-type, edges. In that case, $b(v,v') = 0$ for all $v, v' \in V$ and
the signed adjacency matrix $\dicotmat{K}(G)$ given by Definition~\ref{def:matrix} contains only real entries. This point will come in useful later.

Recall that an {\em automorphism} $\pi$ of a weighted graph $G = [V,E]$ is a bijection from the $V$ to itself which preserves the vertex-edge connectivity as well as vertex- and edge-weights. 

\begin{defn} \label{def:adapted-partition}
Let $G = [V,E]$ be a connected graph and $\pi$ be an automorphism of $G$ such that $\pi$ is a fixed-point free involution.
We say that a partition $\mathcal{P} = \{P_1,P_2\}$ of $V$ with $|P_1| = |P_2|$ is {\bf adapted to $\pi$} if the following conditions hold.
\begin{enumerate}
\item $v \in P_1$ if and only if $\pi(v) \in P_2$,
\item there is no edge between $v$ and $\pi(v)$ for all $v \in P_1$,
\item for each edge $(v,v')$ within $P_1$, the orientation of the edge $(v,v')$ is the same as that of $(\pi(v),\pi(v'))$, and
\item for $v,v' \in P_1$, the edges $(v,\pi(v'))$ and $(v',\pi(v))$ are either both oriented from $P_1$ to $P_2$ or both oriented from $P_2$ to $P_1$.
\end{enumerate}
\end{defn}

Notice that the last condition in the above definition is consistent with our choice of canonical orientation (see Remark~\ref{rem:label}) since we can choose to label the vertices of $P_1$ with $\{1,\dots,n\}$ and those of $P_2$ with $\{n+1,\dots,2n\}$.

\begin{prop}
\label{prop:uniq-adapted}
Let $G = [V,E]$ be a connected graph and $\pi$ be an automorphism of $G$ which is a fixed-point free involution. If a partition adapted to $\pi$ exists, it is unique.
\end{prop}

\begin{proof}
Suppose $\mathcal{P} = \{P_1,P_2\}$ and $\mathcal{P}' = \{P'_1,P'_2\}$ are two distinct partitions adapted to $\pi$. Since $G$ is connected, there must be edges connecting $P_1$ to $P_2$, as well as $P'_1$ to $P'_2$.
Let $P_{i,j} = P_i \cap P'_j$ for $1 \leq i,j \leq 2$. 
We first claim that there must exist vertices $v \in P_{1,1}$ and $w \in P_{1,2}$ which are connected by an edge. If no such pair of vertices exist, then all the edges between $P_1$ and $P_2$ are those between $P_{1,1}$ and $P_{2,2}$, and $P_{1,2}$ and $P_{2,1}$. But this makes $G$ disconnected.

Thus, $(v,w)$ is an edge.
Hence, $\pi(v) \in P_{2,2}$ and $\pi(w) \in P_{2,1}$ are also connected by an edge. Without loss of generality, assume that the $(v,w)$ vertex is oriented from $v$ to $w$. By the adaptedness of $\mathcal{P}$, the $(\pi(v),\pi(w))$ edge is oriented from $\pi(v)$ to $\pi(w)$, but by the adaptedness of $\mathcal{P'}$, it is oriented the other way, which is a contradiction.
\end{proof}

\begin{defn} \label{def:quotient-dicot}
Let $G = [V,E]$ be a connected graph, $\pi$ be an automorphism of $G$ which is a fixed-point free involution and $\mathcal{P}$ be a partition  adapted to $\pi$.
Then the {\bf quotient dicot} corresponding to the pair $(\pi, \mathcal P)$ is  $\dicot{D} = G/\pi$ whose solid edges are given by the subgraph of $G$ restricted to $P_1$ (or equivalently $P_2$) and whose dashed edges are given by those edges $(u,\pi(v)) \in G$ such that $u \in P_1,v \in P_2$.
\end{defn}

Figure~\ref{fig:eg-squarepf-graph-dicots} gives an illustration of a graph and its quotient dicot.

\begin{thm} \label{thm:squarepf}

Let $G$ be a bipartite connected graph with an involution $\pi$ preserving $G$ and a  partition $\mathcal{P} = \{P_1,P_2\}$ of $V$ adapted to $\pi$.
Consider the quotient dicot $\tilde{\dicot{G}}=G/\pi$. If $\tilde{\dicot{G}}$ is also bipartite, then the partition function of the monopole-dimer model on the graph $G$, $Z(G)$, is given by
\[
 Z(G) = \mathcal \dicotpf{Z}(\tilde{\dicot{G}})^2.
\]
\end{thm}

\begin{proof}
We first consider the signed adjacency matrix $K(G)$ with the vertex order $\left(v_1,\dots,v_{n}, \pi(v_1),\dots,\pi(v_{n}) \right)$, where $v_j \in P_1$ and the order of the vertices within $P_1$ is arbitrary. It is natural to write $K(G)$ in $2 \times 2$ block form.
Since $\mathcal{P}$ is adapted to $\pi$, the $(1,1)$ and $(2,2)$ blocks are identical. Moreover, by the fourth condition in Definition~\ref{def:adapted-partition}, the $(1,2)$ block is a symmetric matrix.
Thus the signed adjacency matrix can be written as
\[
K(G) = \left(
\begin{array}{c|c}
M & B \\
\hline\\[-0.3cm]
-B & M
\end{array}
\right).
\]
We now use the identity
\begin{align*}
\left(
\begin{array}{c|c}
\imag/2 \cdot \iden & -\imag/2 \cdot \iden  \\
\hline\\[-0.3cm]
1/2 \cdot \iden  & 1/2 \cdot \iden 
\end{array}
\right)
\cdot
\left(
\begin{array}{c|c}
M-\imag \, B & 0 \\
\hline\\[-0.3cm]
0 & M+\imag \, B
\end{array}
\right)
&\cdot
\left(
\begin{array}{c|c}
-\imag \, \iden & \iden \\
\hline\\[-0.3cm]
\imag \, \iden & \iden
\end{array}
\right) \\
&= 
\left(
\begin{array}{c|c}
M & B \\
\hline\\[-0.3cm]
-B & M
\end{array}
\right),
\end{align*}
and take the determinant. The determinant of the first and last matrices on the right are easily calculated since the block matrices commute, and we obtain $(i/2)^n$ and $(-2i)^n$ respectively. Thus,
\be \label{detkg}
\det K(G) = \det(M + \imag \, B) \det(M -\imag \, B).
\ee
Let $\dicotmat{K}(\tilde{\dicot{G}})$ be the complex adjacency matrix for  $\tilde{\dicot{G}}$ in the same ordered basis $\left(v_1,\dots,v_{n} \right)$. 
We claim that $\dicotmat{K}(\tilde{\dicot{G}}) = M + \imag \, B$. 
To see this, first note that as the diagonal entries of $B$ are zero by the second condition in Definition~\ref{def:adapted-partition}, the monopole weights are as expected. Moreover, the dashed edges arise exactly as given in Definition~\ref{def:quotient-dicot}. 

It therefore remains to show that $\det \dicotmat{K}(\tilde{\dicot{G}}) = \det(M + iB)$ is real. But we have explicitly shown that the determinant of the complex adjacency matrix is always real during the proof of Theorem~\ref{thm:dicot-partn-fn}. Therefore,
\[
\det K(G) = \det \dicotmat{K}(\tilde{\dicot{G}}) \overline{\det \dicotmat{K}(\tilde{\dicot{G}})} = \det \dicotmat{K}(\tilde{\dicot{G}})^2,
\]
which proves the result.
\end{proof}

As an application, Theorem~\ref{thm:squarepf} will be used to justify the squareness of the monopole-dimer model for cycles and rectangular grid proved in \cite{ayyer2015} in Sections~\ref{sec:cycles} and \ref{sec:rect-grids}. However, we first illustrate the theorem by considering an example. 

\begin{eg} \label{eg:squarepf}
Let $G$ be the bipartite graph shown in Figure~\ref{fig:eg-squarepf-graph-dicots} with edge-weights as shown and vertex weights $x_j$ for $1 \leq j \leq 6$ satisfying $x_j = x_{j+3}$ for $1 \leq j \leq 3$. The map $\pi: j \mapsto j+3 \mod 6$ on the vertex set of $G$ then extends naturally to an involution on $G$. The partition $\mathcal{P} = \{\{1,2,3\},\{4,5,6\}\}$ is adapted to $\pi$.

\begin{figure}[h]
\begin{center}
\begin{tabular}{c c}
\begin{tikzpicture} [>=triangle 45, scale = 1.2]
\draw (0.0,0.0) node[circle,inner sep=2pt,draw] {1};
\draw (-0.5, -0.5) node [color = red] {$x_1$};
\draw (0.0,3.0) node[circle,inner sep=2pt,draw] {2};
\draw (-0.5, 3.5) node [color = red] {$x_2$};
\draw (1.0,1.5) node[circle,inner sep=2pt,draw] {3};
\draw (0.5, 1.0) node [color = red] {$x_3$};
\draw (3.0,3.0) node[circle,inner sep=2pt,draw] {4};
\draw (3.5, 3.5) node [color = red] {$x_1$};
\draw (3.0,0.0) node[circle,inner sep=2pt,draw] {5};
\draw (3.5, -0.5) node [color = red] {$x_2$};
\draw (2.0,1.5) node[circle,inner sep=2pt,draw] {6};
\draw (2.5,2.0) node [color = red] {$x_3$};

\draw [-] (0,0.3) -- node [left]  {$a_{12}$} (0,2.7) ;
\draw [-] (0.3,3) -- node [above]  {$b_{12}$} (2.7,3) ;
\draw [-] (3,2.7) -- node [right]  {$a_{12}$} (3,0.3) ;
\draw [-] (0.3,0) -- node [below]  {$b_{12}$} (2.7,0) ;

\draw [-] (0.1,2.75) -- node [near end, left]  {$a_{23}$} (0.8,1.7) ;
\draw [-] (2.9,0.25) -- node [near end, right]  {$a_{23}$} (2.2,1.3) ;
\draw [-] (0.2,2.8) -- node [above]  {$b_{23}$} (1.8,1.7) ;
\draw [-] (2.8,0.2) -- node [below]  {$b_{23}$} (1.2,1.3) ;
\end{tikzpicture}
&
\hspace{1cm}
\begin{tikzpicture} [>=triangle 45, scale = 1.2]
\draw (0.0,0.0) node[circle,inner sep=2pt,draw] {1};
\draw (-0.5, -0.5) node [color = red] {$x_1$};
\draw (0.0,3.0) node[circle,inner sep=2pt,draw] {2};
\draw (-0.5, 3.5) node [color = red] {$x_2$};
\draw (1.0,1.5) node[circle,inner sep=2pt,draw] {3};
\draw (0.5, 1.0) node [color = red] {$x_3$};

\draw [-] (0,0.3) -- node [left]  {$a_{12}$} (0,2.7) ;
\draw [-] (0.1,2.75) -- node [near end, left]  {$a_{23}$} (0.8,1.7) ;
\draw [-,out=135,in=-135,style=dashed] (-0.3,0) to node [left]  {$b_{12}$} (-0.3,3) ;
\draw [-,out=0,in=90,style=dashed] (0.3,3) to node [right]  {$b_{23}$} (1,1.8) ;
\end{tikzpicture}
\end{tabular}

\caption{The graph $G$ and the dicot quotient $\tilde{\dicot{G}}$ of Example~\ref{eg:squarepf} with weights indicated. The vertices are labelled so that the orientation is canonical (see Remark~\ref{rem:label}).}
\label{fig:eg-squarepf-graph-dicots}

\end{center}
\end{figure}

The partition function of the monopole-dimer model on $G$ is given by
\begin{align*}
Z(G) &= \det \left(
\begin{matrix}
x_1 & a_{12} & 0 & 0 & b_{12} & 0 \\
-a_{12} & x_2 & a_{23} & b_{12} & 0 & b_{23} \\
0 & -a_{23} & x_3 & 0 & b_{23} & 0 \\
0 & -b_{12} & 0 & x_1 & a_{12} & 0 \\
-b_{12} & 0 & -b_{23} & -a_{12} & x_2 & a_{23} \\
0 & -b_{23} & 0 & 0 & -a_{23} & x_3
\end{matrix}
\right) \\
&=
\left(x_1 x_2 x_3 + a_{12}^2 x_3 + a_{23}^2 x_1 + b_{12}^2 x_3 + b_{23}^2 x_1\right)^2.
\end{align*}

The dicot $\tilde{\dicot{G}} = G/ \pi$ is shown in Figure~\ref{fig:eg-squarepf-graph-dicots} with vertex set $\{1,2,3\}$.

\begin{align*}
& \dicotpf{Z}(\tilde{\dicot{G}}) =  \det \left(
\begin{matrix}
x_1 & a_{12}+\imag\, b_{12} & 0 \\
-a_{12}+\imag\, b_{12} & x_2 & a_{23}+\imag\, b_{23} \\
0 & -a_{23}+\imag\, b_{23} & x_3 
\end{matrix}
\right) \\
&= x_1 x_2 x_3 + a_{12}^2 x_3 + a_{23}^2 x_1 + b_{12}^2 x_3 + b_{23}^2 x_1.
\end{align*}

\end{eg}

\section{Special families} 
\label{sec:families}

We consider some families of dicots, calculate the partition function of the \mdm{} on them and use the exact results to calculate the free energy.

\subsection{Cycle dicots}
\label{sec:cycles}

Theorem~\ref{thm:squarepf} gives an explanation for the squareness of the monopole-dimer model on the cycle graph $C_{4n}$ with vertex weights $x$ and edge-weights $a$. It was shown in \cite[Example 3.4]{ayyer2015} that
\[
Z(C_{4n}) = \left( a^{2n} L_{2n}(x/a) \right)^2,
\]
where $L_n(x)$ is the Lucas polynomial defined by the recurrence $L_n(x) = x L_{n-1}(x) + L_{n-2}(x)$ with initial conditions $L_0(x) = 2, L_1(x) = x$.

\begin{figure}[htbp!]
\begin{center}
\begin{tabular}{c c}

\begin{tikzpicture} [>=triangle 45, scale=1.2]

\node[draw,minimum size=4cm,regular polygon,regular polygon sides=8] (a) {};

\foreach \x in {1,2,...,4} {
	\node [above left] at (a.corner \x) {$\x$};
	\fill (a.corner \x) circle[radius=2pt];
}

\foreach \x in {5,6,...,8} {
	\node [below right] at (a.corner \x) {$\x$};
	\fill (a.corner \x) circle[radius=2pt];
}

\end{tikzpicture}
&
\hspace{1cm}
\raisebox{.2\height}{
\begin{tikzpicture} [>=triangle 45, scale=1.2]

\node[draw=none,minimum size=4cm,regular polygon,regular polygon sides=4] (a) {};

\foreach \x in {1,2,...,4} {
  \node [above left] at (a.corner \x) {$\x$};
  \fill (a.corner \x) circle[radius=2pt];
}

  \draw [-] (a.corner 1) to (a.corner 2);
  \draw [-] (a.corner 2) to (a.corner 3);
  \draw [-] (a.corner 3) to (a.corner 4);
   \draw [-,dashed] (a.corner 4) to (a.corner 1);
\end{tikzpicture}
}
\end{tabular}

\caption{The cycle graph $C_8$ on the left and the cycle dicot $\dicot C_4$ on the right. In both, the orientations are canonical (see Remark~\ref{rem:label}).}
\label{fig:eg-cycle-dicot}
\end{center}
\end{figure}

Consider the cycle dicot $\dicot C_{2n}$ on vertices $\{1,\dots,2n\}$ where the edges $(j,j+1)$ are solid for $1 \leq j \leq 2n-1$ and the sole dashed edge is $(1,2n)$.
See Figure~\ref{fig:eg-cycle-dicot} for an illustration for the case $n=2$. While this dicot seems to be planar, it doesn't quite fit Definition~\ref{def:planar-dicot}. However, it is a family of graphs that does belong to the classification in Question~\ref{que:positive-pf}.

The partition function for the monopole-dimer model on this cycle dicot is
\[
\dicotpf Z(\dicot C_{2n}) = \det
\begin{pmatrix}
x & a & 0 & \cdots & 0 & \imag \, a \\
-a & \ddots & \ddots & 0 & \cdots & 0 \\
0 & \ddots & \ddots & \ddots & \ddots & \vdots \\
\vdots & \ddots & \ddots & \ddots & \ddots & 0 \\
0 & \cdots & 0 & \ddots & \ddots & a \\
\imag \, a & 0 & \cdots & 0 & -a & x
\end{pmatrix}.
\]
By writing down a recurrence and using identities relating Lucas and Fibonacci polynomials, one can show that the right hand side is given by $a^{2n} L_{2n}(x/a)$ in 
agreement with Theorem~\ref{thm:squarepf}.

It is well-known (see for example \cite{hoggatt-bicknell-1973}) that the Lucas polynomials can be written as
\[
L_n(x) = \prod_{j=0}^{n-1} \left( x - 2 \imag \, \cos \frac{(2j+1) \pi}{2n} \right).
\]
Using this expression, we can calculate the free energy 
\begin{align*}
F(\dicot C) & := \lim_{n \to \infty} \frac{1}{2n} \log a^{2n} L_{2n} \left( \frac{x}{a}
\right) \\
&= \log a + \lim_{n \to \infty} \frac{1}{2n} \sum_{j=0}^{n-1} 
\log \left( \frac{x^2}{a^2} + 4 \cos^2 \frac{(2j+1) \pi}{4n} \right).
\end{align*}
Converting the last sum to a Riemann integral, we obtain
\[
F(\dicot C) = \log a +  \frac{1}{2} \int_0^1 \log \left( \frac{x^2}{a^2} + 4 \cos^2 \frac{\pi t}{2} \right) \text{d}t,
\]
which can be evaluated exactly, leaving us with the simple expression
\[
F(\dicot{C}) = \log \left( \frac{x + \sqrt{x^2 + 4 a^2}}{2} \right). 
\]

\subsection{Rectangular grid graphs}
\label{sec:rect-grids}
Let $Q_{2m,2n}$ be the $2m \times 2n$ rectangular grid graph with horizontal (resp. vertical) edge-weights $a$ (resp. $b$) and vertex-weights $x$. $Q_{2m,2n}$ inherits
the fixed-point free involution $\pi$ which maps vertex $(j,k)$ to $(2m+1-j,2n+1-k)$ for $1 \leq j \leq 2m, 1 \leq k \leq n$.

We consider the orientation $\cO$ on  $Q_{2m,2n}$ given as follows: $(j,k) \to (j+1,k)$ if $k$ is even and the other way if $k$ is odd, 
and $(j,k) \to (j,k+1)$ if $k < n$ and the other way if $k \geq n$. It is easy to see that $\cO$ is a Kasteleyn orientation \cite{kasteleyn1961}.

\begin{figure}[htbp!]
\begin{center}
\begin{tabular}{c c}

\begin{tikzpicture} [>=triangle 45, scale=1.2]

\foreach \x in {0,1,...,3} {
       \foreach \y in {0,1,...,5} {
            \fill[color=black] (\x,\y) circle (0.05);
       }
}
\draw (3,0) -- (2,0) -- (1,0) -- (0,0);
\draw (0,1) to (1,1) to (2,1) to (3,1);
\draw (3,2) -- (2,2) -- (1,2) -- (0,2);
\draw (0,3) to (1,3) to (2,3) to (3,3);
\draw (3,4) -- (2,4) -- (1,4) -- (0,4);
\draw (0,5) to (1,5) to (2,5) to (3,5);

\draw (0,2) -- (0,1) -- (0,0);
\draw (0,2) -- (0,3) -- (0,4) -- (0,5);

\draw (1,2) -- (1,1) -- (1,0);
\draw (1,2) -- (1,3) -- (1,4) -- (1,5);

\draw (2,2) -- (2,1) -- (2,0);
\draw (2,2) -- (2,3) -- (2,4) -- (2,5);

\draw (3,2) -- (3,1) -- (3,0);
\draw (3,2) -- (3,3) -- (3,4) -- (3,5);

\node [below left] at (0,0) {$1$};
\node [below left] at (1,0) {$2$};
\node [below left] at (2,0) {$3$};
\node [below left] at (3,0) {$4$};

\node [below left] at (0,1) {$8$};
\node [below left] at (1,1) {$7$};
\node [below left] at (2,1) {$6$};
\node [below left] at (3,1) {$5$};

\node [below left] at (0,2) {$9$};
\node [below left] at (1,2) {$10$};
\node [below left] at (2,2) {$11$};
\node [below left] at (3,2) {$12$};

\node [below left] at (0,3) {$16$};
\node [below left] at (1,3) {$15$};
\node [below left] at (2,3) {$14$};
\node [below left] at (3,3) {$13$};

\node [below left] at (0,4) {$17$};
\node [below left] at (1,4) {$18$};
\node [below left] at (2,4) {$19$};
\node [below left] at (3,4) {$20$};

\node [below left] at (0,5) {$24$};
\node [below left] at (1,5) {$23$};
\node [below left] at (2,5) {$22$};
\node [below left] at (3,5) {$21$};

\end{tikzpicture}
&
\hspace{1cm}
\raisebox{.4\height}{
\begin{tikzpicture} [>=triangle 45, scale=1.2]

\foreach \x in {0,1,...,3} {
       \foreach \y in {0,1,...,2} {
            \fill[color=black] (\x,\y) circle (0.05);
       }
}

\draw (3,0) -- (2,0) -- (1,0) -- (0,0);
\draw (0,1) to (1,1) to (2,1) to (3,1);
\draw (3,2) -- (2,2) -- (1,2) -- (0,2);

\draw (0,2) -- (0,1) -- (0,0);
\draw (1,2) -- (1,1) -- (1,0);
\draw (2,2) -- (2,1) -- (2,0);
\draw (3,2) -- (3,1) -- (3,0);

\node [below left] at (0,0) {$1$};
\node [below left] at (1,0) {$2$};
\node [below left] at (2,0) {$3$};
\node [below left] at (3,0) {$4$};

\node [below left] at (0,1) {$8$};
\node [below left] at (1,1) {$7$};
\node [below left] at (2,1) {$6$};
\node [below left] at (3,1) {$5$};

\node [below left] at (0,2) {$9$};
\node [below left] at (1,2) {$10$};
\node [below left] at (2,2) {$11$};
\node [below left] at (3,2) {$12$};

\draw [out = 90,in = 90,dashed] (0,2) to (3,2);
\draw [out = 90,in = 90,dashed] (1,2) to (2,2);

\end{tikzpicture}
}
\end{tabular}

\caption{Illustration of $Q_{4,6}$ on the left and $\tilde{\dicot{Q}}_{4,3}$ on the right. The vertices are labelled so that the orientation is canonical (see Remark~\ref{rem:label}).}
\label{fig:eg-dicot-q23}
\end{center}
\end{figure}

For the monopole-dimer model on the $2m \times 2n$ rectangular grid, the partition function had an explicit form. 

\begin{thm}[{\cite[Theorem 4.1]{ayyer2015}}] 
\label{thm:partfngrid}
The partition function of the monopole-dimer model on the graph 
$Q_{2m,2n}$ is given by
\[
\prod_{j=1}^{m}  \prod_{k=1}^{n}  
\left( x^2 + 4 a^2 \cos^2 \left( \frac{j \pi}{2m+1} \right) + 4 b^2 \cos^2 \left( \frac{k \pi}{2n+1} \right) \right)^2.
\]
\end{thm}

Then one can check from Definition~\ref{def:adapted-partition} that $\mathcal{P} = \{P_1, P_2\}$ with
\begin{align*}
P_1 &= \{(j,k) \mid 1 \leq j \leq 2m, 1 \leq k \leq n \} \quad \text{and} \\
P_2 &= \{(j,k) \mid 1 \leq j \leq 2m, n+1 \leq k \leq 2n \},
\end{align*}
is an adapted partition.
Then define the quotient dicot $\tilde{\dicot{Q}}_{2m,n} = Q_{2m,2n}/\pi$ by Definition~\ref{def:quotient-dicot}.  
The dicots $\tilde{\dicot{Q}}_{2m,n}$ are again not planar because they fail to be pure (see Definition~\ref{def:pure-dicot}) at a single edge. 
An illustration of the dicot $\tilde{\dicot{Q}}_{4,3}$ is given in Figure~\ref{fig:eg-dicot-q23}.
We are now in a position to apply Theorem~\ref{thm:squarepf} and we thus obtain the following corollary.

\begin{cor}
\label{cor:partfn-grid-dicot}
The partition function of the monopole-dimer model on the dicot $\tilde{\dicot{Q}}_{2m,n}$ is given by
\[
\prod_{j=1}^{m}  \prod_{k=1}^{n}  
\left( x^2 + 4 a^2 \cos^2 \left( \frac{j \pi}{2m+1} \right) + 4 b^2 \cos^2 \left( \frac{k \pi}{2n+1} \right) \right).
\]
\end{cor}

The computation of the free energy $F(\tilde{\dicot{Q}})$ has already been performed in \cite[Section 5]{ayyer2015}.

\subsection{Rectangular grid dicots}
\label{sec:rect-grids-add-vert}

We now consider the monopole-dimer model on the dicot $\dicot Q^{(v)}_{m,n}$, which has the same vertex set as and includes all the edges of $Q_{m,n}$ but has in addition a dashed edge for every vertical solid edge. Horizontal edges are assigned weight $a$, vertical solid edges, $b_1$, vertical dashed edges, $b_2$, and vertices $x$.
Since $\dicot Q^{(v)}_{m,n}$ are not pure by Definition~\ref{def:pure-dicot}, they are non-planar.
Note that not all configurations have positive weight; see the example in Figure~\ref{fig:eg-vert-dicot-q23}. Set $|b| = \sqrt{b_1^2 + b_2^2}$.

\begin{figure}[htbp!]
\begin{tabular}{l r}
\begin{tikzpicture} [>=triangle 45, scale=1.5]

\foreach \x in {0,1,...,3} {
       \foreach \y in {0,1,...,2} {
            \fill[color=black] (\x,\y) circle (0.05);
       }
}
\draw [-] (0,0) -- (3,0);
\draw [-] (0,1) -- (3,1);
\draw [-] (0,2) -- (3,2);

\draw [-] (0,0) -- (0,2);
\draw [-] (1,0) -- (1,2);
\draw [-] (2,0) -- (2,2);
\draw [-] (3,0) -- (3,2);

\node [below left] at (0,0) {$1$};
\node [below left] at (1,0) {$2$};
\node [below left] at (2,0) {$3$};
\node [below left] at (3,0) {$4$};

\node [below left] at (0,1) {$8$};
\node [below left] at (1,1) {$7$};
\node [below left] at (2,1) {$6$};
\node [below left] at (3,1) {$5$};

\node [below left] at (0,2) {$9$};
\node [below left] at (1,2) {$10$};
\node [below left] at (2,2) {$11$};
\node [below left] at (3,2) {$12$};

\draw [out = 30,in = -30,dashed] (0,0) to (0,1);
\draw [out = 30,in = -30,dashed] (0,1) to (0,2);
\draw [out = 30,in = -30,dashed] (1,0) to (1,1);
\draw [out = 30,in = -30,dashed] (1,1) to (1,2);
\draw [out = 30,in = -30,dashed] (2,0) to (2,1);
\draw [out = 30,in = -30,dashed] (2,1) to (2,2);
\draw [out = 30,in = -30,dashed] (3,0) to (3,1);
\draw [out = 30,in = -30,dashed] (3,1) to (3,2);

\end{tikzpicture}

&

\hspace{0.5cm}
\begin{tikzpicture} [>=triangle 45, scale=1.5]

\foreach \x in {0,1,...,3} {
       \foreach \y in {0,1,...,2} {
            \fill[color=black] (\x,\y) circle (0.05);
       }
}

\node [below left] at (0,0) {$1$};
\node at (1,0) [circle,inner sep=3pt,fill= OliveGreen, label=below left:2] {};
\node [below left] at (2,0) {$3$};
\node [below left] at (3,0) {$4$};

\node [below left] at (0,1) {$8$};
\node [below left] at (1,1) {$7$};
\node [below left] at (2,1) {$6$};
\node [below left] at (3,1) {$5$};

\node at (0,2) [circle,inner sep=3pt,fill= OliveGreen, label=below left:9] {};
\node [below left] at (1,2) {$10$};
\node [below left] at (2,2) {$11$};
\node [below left] at (3,2) {$12$};

\draw [draw = blue,-] (-0.05,0) -- (-0.05,1)
		[xshift=1pt] (0,0) -- (0,1) ;

\draw [draw = red,-] (2,0) -- (3,0);
\draw [draw = red,-] (2,0) -- (2,1);
\draw [draw = red,-] (1,1) -- (2,1);
\draw [draw = red,-] (1,1) -- (1,2);
\draw [draw = red,-] (1,2) -- (3,2);

\draw [draw = red,out = 30,in = -30,dashed] (3,0) to (3,1);
\draw [draw = red,out = 30,in = -30,dashed] (3,1) to (3,2);

\end{tikzpicture}
\end{tabular}

\caption{Illustration of the dicot $\dicot Q^{(v)}_{4,3}$ on the left and a particular \md{} configuration on the right with weight $-a^4 \, b_1^4 \, b_2^2 \, x^2$.}
\label{fig:eg-vert-dicot-q23}
\end{figure}

Define the function
\[
Y_m(b,c;x) = \prod_{j=1}^{\lfloor m/2 \rfloor} \left( x^2 + 4 (b^2+c^2) 
\cos^2 \left( \frac{j \pi}{m+1} \right) \right).
\]
Then we have the following result.

\begin{thm} \label{thm:partfn-grid-vert-dicot}
The partition function of the monopole-dimer model on $\dicot Q^{(v)}_{m,n}$ is given by
\begin{equation}	
\begin{split} \label{partfn-grid}
\dicotpf Z^{(v)}_{{m,n}}=& \prod_{j=1}^{\lfloor m/2 \rfloor}  \prod_{k=1}^{\lfloor n/2 \rfloor}  
\left( x^2 + 4 a^2 \cos^2 \left( \frac{j \pi}{m+1} \right) + 4 |b|^2 
\cos^2 \left( \frac{k \pi}{n+1} \right) \right)^2 \\
&\times \begin{cases}
1 & \text{if $m$ and $n$ are even}, \\
Y_m(a,0;x) & \text{if $m$ is even and $n$ is odd}, \\
Y_n(b_1,b_2;x) & \text{if $m$ is odd and $n$ is even}, \\
x Y_m(a,0;x) Y_n(b_1,b_2;x) & \text{if $m$ and $n$ are odd}.
\end{cases}
\end{split}
\end{equation} 
\end{thm}

\begin{proof}
The strategy of diagonalisation is very similar to that of \cite[Theorem 4.1]{ayyer2015}. The sole difference is in the matrices involved in the diagonalisation process. Consider the $n \times n$ tridiagonal Toeplitz matrix given by
\[
T_n(c_{-1},c_0,c_1) = 
\begin{pmatrix}
c_0 & c_1 & 0 & \cdots & 0 & 0 \\
c_{-1} & \ddots & \ddots & 0 & \cdots & 0 \\
0 & \ddots & \ddots & \ddots & \ddots & \vdots \\
\vdots & \ddots & \ddots & \ddots & \ddots & 0 \\
0 & \cdots & 0 & \ddots & \ddots & c_1 \\
0 & 0 & \cdots & 0 & c_{-1} & c_0
\end{pmatrix}.
\]
It is well-known (see for example \cite{trench-1985}) that the eigenvalues of \\ $T_n(c_{-1},c_0,c_1)$ are
\[
\lambda_q = c_0 + 2\sqrt{c_1 c_{-1}} \cos \frac{q \pi}{n+1}, \quad q=1,\dots,n,
\]
and the matrix of column eigenvectors is
\[
u_n(c_{-1},c_1)_{j,q} = \left( \frac{c_{-1}}{c_1} \right)^{(j-1)/2} \sin \frac{q j \pi}{n+1}, \quad j=1,\dots,n,
\]
respectively. Let $J_m$ be the antidiagonal matrix of ones of size $m$.
Following Kasteleyn~\cite{kasteleyn1961}, the complex adjacency matrix with the standard ordering (explained in the beginning of Section~\ref{sec:rect-grids}) can be written using tensor notation as 
\[
\dicotmat{K}^{(v)}_{m,n} = \id{n} \otimes T_m(-a,x,a) + T_n(-b_1+\imag \, b_2,0,b_1+\imag \, b_2) \otimes J_m.
\]
Now we perform a similarity transformation using $u_n(-b_1+\imag \, b_2,b_1+\imag \, b_2) \otimes u_m(-a,a)$ leading to
a block diagonal matrix whose $k$'th block is the cruciform matrix (i.e. one whose nonzero entries lie only on the diagonal and antidiagonal)
\[
\begin{pmatrix}
x + 2\imag \, a \cos \frac{\pi}{m+1} &  & & (-1)^{m-1} 2 \imag^m \, |b| \cos \frac{\pi k}{n+1} \\
\vspace{-0.2cm}  &  \ddots & \Ddots & \\
 &  \Ddots & \ddots & \\
(-1)^{0} 2 \imag^m \, |b| \cos \frac{\pi k}{n+1} & & & x + 2\imag \, a \cos \frac{m \pi}{m+1}
\end{pmatrix}.
\]
The determinant of the block-diagonal matrix is thus easily computed. The four cases where $m$ and $n$ are of different parities are handled exactly as in the proof of \cite[Theorem 4.1]{ayyer2015}.
\end{proof}

When we set $b_1=0$, we obtain a planar dicot for which we can define a Kasteleyn orientation using Proposition~\ref{prop:kast-orient-exist}. This then matches the orientation defined by Wu \cite{wu1962} (see Remark~\ref{rem:wu-orientation}).

The free energy of the monopole-dimer model on the infinite grid dicot with vertical dashed edges $\lim_{m,n \to \infty} \dicot Q^{(v)}_{m,n}$ is then given by
\begin{align*}
F(\dicot Q^{(v)}) &= \lim_{m,n \to \infty} \frac 1{mn} \log \dicotpf Z^{(v)}_{{m,n}}\\
 &= \frac{2}{\pi^2} \int_0^{\pi/2} \text{d} \theta \int_0^{\pi/2} 
\text{d} \phi \; \ln \left( x^2 + 4 a^2 \cos^2 \theta + 4 |b|^2 \cos^2 \phi \right),
\end{align*}
and the results are analogous to that of the monopole-dimer model on the grid graph 
described in \cite[Section 5]{ayyer2015} with $b$ there replaced by $|b|$.

\subsection{Wheel dicots}
\label{sec:wheel}
For $n$ an odd integer, define the {\em wheel dicot of order $n$}, denoted $\dicot{W}_n$, to be a dicot on the vertex set $V = [2n]$ with solid edges such that the graph $[V, A]$ is the ordinary cycle graph $C_{2n}$ and dashed edges connecting antipodal vertices.
It is easy to see that $\dicot{W}_n$ 
is indeed a dicot (see Definition~\ref{def:dicot}). The underlying graph $[V, A \cup B]$ of $\dicot W_3$ is isomorphic to $K_{3,3}$. Thus, for $n \geq 3$, $\dicot{W}_n$ is non-planar.

\begin{figure}[h]
\begin{tabular}{c c}
\begin{tikzpicture} [>=triangle 45, scale = 2]

\draw (5.3,4.0) node [below] {1} 
			-- (4.7,4.0) node [below] {9} 
			-- (4.2,4.4) node [left] {2} 
			-- (4.0,5.0) node [left] {10} 
			-- (4.2,5.6) node [left] {3} 
			-- (4.7,6.0) node [above] {6} 
			-- (5.3,6.0) node [above] {4} 
			-- (5.8,5.6) node [right] {7} 
			-- (6.0,5.0) node [right] {5} 
			-- (5.8,4.4) node [right] {8} 
			-- cycle ;
\draw [-,dashed]	 (5.3,4.0) -- (4.7,6.0);
\draw [-,dashed]	 (4.2,4.4) -- (5.8,5.6);
\draw [-,dashed]	 (4.2,5.6) -- (5.8,4.4);
\draw [-,dashed]	 (5.3,6.0) -- (4.7,4.0);
\draw [-,dashed]	 (6.0,5.0) -- (4.0,5.0);

\end{tikzpicture}

&

\hspace{0.5cm}
\begin{tikzpicture} [>=triangle 45, scale = 2]

\draw [draw = red, thick] (5.3,4.0) node [below] {1} 
			-- (4.7,4.0) node [below] {9}; 
\draw [draw = red, thick] (4.7,6.0) node [above] {6} 
			-- (5.3,6.0) node [above] {4} ; 
\draw [draw = red,-,dashed, thick]	 (5.3,4.0) -- (4.7,6.0);
\draw [draw = red,-,dashed, thick]	 (5.3,6.0) -- (4.7,4.0);

\draw [draw = blue,-,dashed, thick]	 (4.2,5.6) -- (5.8,4.4)
			[xshift=1pt,yshift=1pt] (4.2,5.6) -- (5.8,4.4);
			
\draw [draw = Orange, thick] (4.2,4.4) node [left] {2} 
			-- (4.0,5.0) node [left] {10}
			[xshift=1pt,yshift=1pt] (4.2,4.4) -- (4.0,5.0); 

\node at (5.8,5.6) [circle,inner sep=3pt,fill= OliveGreen, label=right:7] {};
\node at (6.0,5.0) [circle,inner sep=3pt,fill = OliveGreen, label=right:5] {};

\draw (4.2,5.6) node [left] {3};
\draw (5.8,4.4) node [right] {8};

\end{tikzpicture}

\end{tabular}

\caption{The wheel dicot $\dicot{W}_{5}$ on the left and a particular \mdm{} configuration on the right with weight $x^2 a^4 b^4$.}
\label{fig:eg-wheel-dicot}
\end{figure}

We now label the vertices of $\dicot{W}_n$ as follows starting by labelling an arbitrary vertex 1. Proceeding clockwise, we label every alternate vertex with the next integer until we assign label $n$. We now label the vertex antipodal to 1 by $n+1$. Similarly, proceeding clockwise, we label every alternate vertex with the next integer until we assign label $2n$. At this point we have labelled every vertex. See Figure~\ref{fig:eg-wheel-dicot} for the dicot $\dicot{W}_5$ with its labelling. 
We then assign the canonical orientation on solid edges (see Remark~\ref{rem:label}). Thus, each vertex is  oriented either outwards or inwards on both solid edges.

Consider the \mdm{} on $\dicot{W}_n$ for odd $n$, where vertex-weights are $x$, solid edge-weights are $a$ and dashed edge-weights are $b$. Then we have the following result.

\begin{lem}
\label{lem:wheel-dicots-positive-wt}
For odd $n$, every \md{} configuration on $\dicot{W}_n$ has a positive weight.
\end{lem}

\begin{proof}
Since the weights of doubled edges and isolated vertices have positive sign,
it suffices to prove that every loop has positive weight. We will divide the proof in two parts. We will first show that every nontrivial loop involving dashed edges is {\em $8$-shaped} with exactly two dashed edges. 
An $8$-shaped loop is one involving exactly two dashed edges.
In the example \mdm{} configuration in Figure~\ref{fig:eg-wheel-dicot}, $(1,6,4,9,1)$ is an $8$-shaped loop.
In the second part, we will show that all loops have positive weight. 

The first statement follows from the fact that every loop contains an even number of dashed edges. For example, suppose we have a loop in $\dicot{W}_5$ of Figure~\ref{fig:eg-wheel-dicot} starting with $6-4-9-2-\cdots$. To eventually get back to 6, we need an even number of dashed edges to reach the same side of the $4-9$ dashed edge as 6. But this would imply that we have an odd number of total dashed edges. Thus, a loop starting 
$6-4-9$ has to extend as $6-4-9-1-\cdots$. It is now easy to see that there has to be exactly one more dashed edge in the loop.

Recall the formula for the weight of a loop of length $2m$, \eqref{wtloop}.
First consider (trivial) loops of length 2 with $m=1$. Such loops have positive weight in both solid and dashed cases. Next, consider loops without dashed edges. The only such loops without dashed edges are the circumferences. In this case, $m=n$, and there are $n$ negative orientations, leading to a positive weight. 

Lastly, for $8$-shaped loops of length $2m$, we have a prefactor of $-1$ in the definition of the weight in \eqref{wtloop} and a factor of $-1$ from the two dashed edges, leading to an overall factor of $+1$. 
We thus have to show that the sign coming purely from orientations of solid edges is $+1$. First, note that will have $m-1$ solid edges on either side. 
If $m$ is odd, then there are $(m-1)/2$ edges of opposite orientation on either side leading to a total sign of $(-1)^{m-1} = +1$. If $m$ is even, we have either $(m-2)/2$ or $m/2$ edges with opposite orientation on one side. But because dashed edges connect vertices on different (bipartite) parts, we will have the same number of edges with opposite orientation on the other side too, leading to a total sign of either $(-1)^{m-2}$ or $(-1)^m$, which is again $+1$.
\end{proof}

\begin{thm}
\label{thm:wheel-dicot-partn-fn}
The partition function of the \mdm{} on the wheel dicot $\dicot{W}_n$ for odd $n$ is given by
\begin{align*}
\dicotpf{Z}(\dicot{W}_n) &= \prod_{j=0}^{n-1} \left( x^2 + b^2 + 4 a^2 \cos^2 \frac{\pi j}{n} \right) \\
&= (x^2 + b^2 + 4 a^2) \prod_{j=1}^{(n-1)/2} 
\left( x^2 + b^2 + 4 a^2 \cos^2 \frac{\pi j}{n} \right)^2.
\end{align*}
\end{thm}

\begin{proof}
We will use Theorem~\ref{thm:dicot-partn-fn} to calculate $\dicotpf{Z}(\dicot{W}_n)$. The complex adjacency matrix $\dicotmat{K}(\dicot{W}_n)$ using the natural order of the vertices is given by
\[
\dicotmat{K}(\dicot{W}_n) =
\left(
\begin{array}{c|c}
x \iden & a A + \imag \, b \iden \\
\hline\\[-0.3cm]
-a A + \imag \, b \iden & x \iden
\end{array}
\right),
\]
where $A$ is an $n \times n$ symmetric $(0,1)$-band matrix consisting of two consecutive bands above and below the diagonal equidistant from both ends. Since both the matrices in the top-blocks commute, the determinant of $\dicotmat{K}(\dicot{W}_n)$ can be written in a way analogous to a $2\times 2$ matrix as
\begin{equation}
\label{wheeldicot-det}
\begin{split}
\det \dicotmat{K}(\dicot{W}_n) 
&= \det(x \iden \cdot x \iden - (a A + \imag \, b \iden) \cdot (-a A + \imag \, b \iden)) \\
&= \det((x^2 + b^2) \iden + a^2 A^2).
\end{split}
\end{equation} 
It now suffices to calculate the eigenvalues of $A^2$. But this is easily done since $A^2$ is of the form
\[
A^2 = 
\begin{pmatrix}
2 & 1 & 0 & \cdots & 0 & 1 \\
1 & \ddots & \ddots & 0 & \cdots & 0 \\
0 & \ddots & \ddots & \ddots & \ddots & \vdots \\
\vdots & \ddots & \ddots & \ddots & \ddots & 0 \\
0 & \cdots & 0 & \ddots & \ddots & 1 \\
1 & 0 & \cdots & 0 & 1 & 2
\end{pmatrix},
\]
and hence is a circulant matrix. Thus, the eigenvalues of $A^2$ are given by
\[
\lambda_j = 2 + \omega_j + \omega_j^{n-1}, \qquad \text{for $j=0,1,\dots,n-1$},
\]
where $\omega_j = \exp(2 \pi \imag \, j /n)$ is the $n$'th root of unity. It is easy to see that 
\[
\lambda_j = 4 \cos^2 \frac{\pi j}{n}.
\]
Since the determinant of $\dicotmat{K}(\dicot{W}_n)$ is the product of its eigenvalues, we use this expression in \eqref{wheeldicot-det} to complete the proof.
\end{proof}

From Theorem~\ref{thm:wheel-dicot-partn-fn}, one can also calculate the free energy of the \mdm{} for wheel dicots. Since the relevant parameter is the ratio of $4a^2$ and $x^2 + b^2$, we set the former to $\alpha$ and the latter to $1$. The free energy is 
\begin{align*}
F(\dicot{W}) &:= \lim_{n \to \infty} \frac{1}{n} \log \dicotpf{Z}(\dicot{W}_n), \\
&= \lim_{n \to \infty} \frac{1}{n} \sum_{j=0}^{n-1} 
\log \left(1 + \alpha \cos^2 \frac{\pi j}{n} \right).
\end{align*}
We replace the right hand side by the Riemann integral to obtain
\[
F(\dicot{W}) = \int_0^1 \log \left(1 + \alpha \cos^2 (\pi t) \right) \text{d}t,
\]
which is essentially the same integral that we saw for the free energy of cycle dicots in Section~\ref{sec:cycles}. Performing the integral, we obtain
\[
F(\dicot{W}) = 2 \log \left( \frac{1 + \sqrt{1 + \alpha}}2 \right).
\]

\section*{Acknowledgements}
We thank Jeremie Bouttier for discussions and anonymous referees for many useful comments.
This work is supported in part by the UGC Centre for Advanced Studies and by Department of Science and Technology grants
DST/INT/SWD/VR/P-01/2014 and EMR/2016/006624.

\end{document}